\documentclass
[12pt]{article}

\usepackage{amssymb}
\usepackage{amsmath}
\usepackage{amsthm}

\usepackage{enumerate}

\usepackage{graphicx}

\newtheorem{thm}{Theorem}[section]
\newtheorem{cor}[thm]{Corollary}

\newtheorem{pro}[thm]{Proposition}

\newtheorem{alg}[thm]{Algorithm}

\theoremstyle{definition}

\newtheorem{ex}[thm]{Example}

\numberwithin{equation}{section}

\newcommand{\Z}{\mathbb{Z}}
\newcommand{\Q}{\mathbb{Q}}
\newcommand{\R}{\mathbb{R}}

\newcommand{\ord}[2]{\nu_{#1}\left(#2\right)}

\begin{document}


\baselineskip=17pt


\title{A Division Algorithm Approach to $p$-Adic Sylvester Expansions}

\author{Eric Errthum\\
Department of Mathematics and Statistics\\ 
Winona State University\\
Winona, MN\\
E-mail: eerrthum@winona.edu\\
}

\date{}

\maketitle


\renewcommand{\thefootnote}{}



\renewcommand{\thefootnote}{\arabic{footnote}}
\setcounter{footnote}{0}


\begin{abstract}
A method of constructing finite $p$-adic Sylvester expansions for all rationals is presented. This method parallels the classical Fibonacci-Sylvester (greedy) algorithm by iterating a $p$-adic division algorithm. The method extends to irrational $p$-adics that have an embedding in the reals.
\end{abstract}

\begin{section}{Introduction}
In \cite{KK1} and \cite{KK2} A. Knopfmacher and J. Knopfmacher give algorithms for constructing Egyptian fraction expansions in a $p$-adic setting that are analogous to those given for the reals by Oppenheim \cite{Oppen}. 
However, for some positive rational inputs the Knopfmachers' algorithm fails to return a finite expansion. In Section 2 we review the basics of the Knopfmachers' Sylvester-type algorithm and give such an example. We then introduce a modification of their algorithm that will give finite Sylvester expansions for all rationals. The seemingly unnatural correction term in our algorithm is explained by the alternate approach detailed in the final two sections.

In Section 3 we briefly recall the Fibonacci-Sylvester Greedy Algorithm, especially its relationship to the classical division algorithm. This provides the main motivation for Section 4 wherein we define a new $p$-adic division algorithm and use the same relationship to construct finite rational $p$-adic Sylvester expansions. Lastly we show that the division algorithm approach and the given modification to Knopfmachers' algorithm yield the same output.
\end{section}

\begin{section}{$p$-Adic Numbers}\label{padic}
\subsection{
Basics of $p$-Adic Numbers}
We begin with some of the necessary basics of $p$-adic numbers. A more thorough exposition can be found in \cite{Koblitz}. 
Let $p$ be a prime and $\Q_p$ the completion of the rationals with respect to the $p$-adic absolute value $|\!\cdot\! |_p$ defined on $\Q$ by 
$$|0|_p = 0 \text{ and } |r|_p = p^{-\nu} \text{ if } r=\frac{a}{b}p^\nu\text{, where }p \nmid ab.$$
The exponent $\nu \in \Z$ is the $p$-adic valuation, or order, of $r$ and will be denoted $\ord{p}{r}$. 

\begin{pro}[c.f. \cite{Koblitz}]\label{Basics} Let nonzero $\zeta,\xi \in \Q_p$.
\begin{enumerate}
\item There exists a unique $\widehat{\zeta} \in \Q_p$ such that $\ord{p}{\widehat{\zeta}} = 0$ and $\zeta=\widehat{\zeta}p^{\ord{p}{\zeta}}$. We will call $\widehat{\zeta}$ the unit part of $\zeta$.
\item $\ord{p}{\zeta\xi}=\ord{p}{\zeta}+\ord{p}{\xi}$.
\item $\ord{p}{\zeta+\xi}\geq\text{min}\{\ord{p}{\zeta},\ord{p}{\xi}\}$.  We have equality when $\ord{p}{\zeta}\neq\ord{p}{\xi}$.
\end{enumerate}
\end{pro}
Every element of $\Q_p$ has the shape \begin{eqnarray}\label{shape}\zeta = \sum_{n=\ord{p}{\zeta}}^\infty c_n p^n \end{eqnarray} for $c_n \in \{0,1, \dots, p-1\}$. 
This representation is unique, however $p$-adic numbers can also be represented in the form 
\begin{eqnarray}
\zeta = a_0 + \sum_{n=1} \frac{1}{a_n} \label{EgyptSeries}
\end{eqnarray}
where $a_n \in \Z[\tfrac{1}{p}]$ and the sum is either finite or converges $p$-adically. There are a variety of algorithms to decompose a $p$-adic into this form (c.f. \cite{KK1}). One reason to study such representations is that they can be considered the $p$-adic analogue to representing real numbers as Egyptian fractions, i.e. as the sum of unit fractions.  

\subsection{
Sylvester-type Series Expansions of $p$-Adic Numbers}

There are various methods of decomposing a rational number into an Egyptian fraction representation. One of the more na{\"i}ve methods was given originally by Fibonacci and again in modern times by Sylvester \cite{Sylv}. This algorithm is  commonly known as the Greedy Algorithm because at each inductive step of the decomposition one simply takes the largest unit fraction smaller than the value being decomposed. Later, Oppenheim \cite{Oppen} generalized this and other Egyptian fraction algorithms to real numbers given by their decimal representations.

The following inductive algorithm, which results in an expansion like (\ref{EgyptSeries}), is presented by the Knopfmachers in \cite{KK1} as a $p$-adic analogue to the Sylvester-Oppenheim algorithm on real numbers. Begin by defining the fractional part of a $p$-adic number $\zeta$ as in (\ref{shape}) by 
\begin{eqnarray}\label{bracket}\langle \zeta \rangle = \sum_{n=\ord{p}{\zeta}}^0 c_n p^n.\end{eqnarray} 

\begin{alg}[\cite{KK1}]\label{Knopf} Let $\zeta \in \Q_p$. For the initial term, set $a_0 = \langle \zeta \rangle 
$. Then take $\zeta_1 = \zeta - a_0$ so that $\ord{p}{\zeta_1} \ge 1$. Continuing as long as $\zeta_n \neq 0$, let $$a_n = \left\langle \frac{1}{\zeta_n}\right\rangle \text{ and } \zeta_{n+1} = \zeta_n - \frac{1}{a_n}.$$ \end{alg}

In \cite{KK2} (see Proposition 5.3) it is stated that the above Sylvester-type algorithm and a similarly defined Engel-type algorithm terminate if and only if $\zeta \in \Q$. However  in \cite{GrabK} Grabner and A. Knopfmacher give an example of a rational with nonterminating Engel-type algorithm. Likewise, there are rationals for which this Sylvester-type algorithm does not terminate.

Indeed, suppose $\zeta = \frac{a}{p+a} \in \Q$ with $p \nmid a$. Then $a_0=1$ and $\zeta_1 < 0$. Since all the $a_n$ are defined to be positive, a finite sum in (\ref{EgyptSeries}) leads to a contradiction. 
Although Laohakosol and Kanasri \cite{LK} give a complete characterization of the infinite Sylvester-type expansions from the Knopfmachers' algorithm that correspond to rational numbers, it is less obvious the necessary conditions under which a rational will result in a finite expansion. 

\subsection{A New Sylvester-Type Algorithm}
We start by generalizing the definition in (\ref{bracket}) in the following way. Let
$$\langle \zeta \rangle_k = \sum_{n=\ord{p}{\zeta}}^{k-1} c_n p^n$$
so that $\langle \zeta \rangle_1=\langle \zeta \rangle$. In other words, $\langle \zeta \rangle_k$ is the rational image of $\zeta$ under the $\mod p^k$ projection from $\Q_p$ to $\Z[\frac{1}{p}]$.

For real $x$, define the ceiling  function, $\lceil x \rceil$, to be the least integer greater than $x$. Notice that since $\Q_p$ is not an ordered field, this function is not well-defined for all $p$-adics. However, if $\zeta \in \Q_p$ such that there exists an embedding of $\Q(\zeta)$ into $\mathbb{R}$, then the ceiling function pulls back to $\Q(\zeta) \subset \Q_p$. 

We now state our modification to Algorithm \ref{Knopf}.
\begin{alg}\label{padicSyl}
Let  nonzero $\zeta \in \Q_p$ such that there exists an embedding $\psi:\Q(\zeta) \to \R$ and $k\in \Z$ such that $k>-\ord{p}{\zeta}$. Set $\zeta_0=\zeta$. Inductively for $i \ge 0$ set $t_i=\left\langle \frac{1}{\zeta_i}\right\rangle_k$, $$q_i=t_i+\left\lceil\frac{1-t_i\psi(\zeta_i)}{p^k\psi(\zeta_i)}\right\rceil p^k,$$ and \begin{align}\label{induct}\zeta_{i+1} = \zeta_i-\frac{1}{q_i}.\end{align} The algorithm terminates if any $\zeta_N=0$.
\end{alg}
\begin{thm}\label{converges}
Algorithm \ref{padicSyl} produces a sequence of $q_i \in \Z[\frac{1}{p}]$ such that $$\zeta = \sum_{i=0} \frac{1}{q_i}$$ where the sum (if infinite) converges $p$-adically.
\begin{proof}
By (\ref{induct}), if the sum converges, it does so to $\zeta$. It suffices to show that $|\zeta_i|_p \to 0$.

Suppose $\ord{p}{\zeta_i}=s>-k$ so that $\zeta_i = \widehat{\zeta_i}p^s$.  Then $$t_i = \left\langle \frac{1}{\widehat{\zeta_i}}p^{-s}\right\rangle_k = \left\langle \widehat{\zeta_1}^{-1}\right\rangle_{k+s} p^{-s}$$
and
$$q_i =  \left( \left\langle \widehat{\zeta_1}^{-1}\right\rangle_{k+s}+mp^{k+s}\right)p^{-s}$$
for some $m \in \Z$. Then $$\zeta_iq_i-1 = \widehat{\zeta_i}\left\langle \widehat{\zeta_1}^{-1}\right\rangle_{k+s}+\widehat{\zeta_i}mp^{k+s}-1 \equiv 0 \mod p^{k+s}$$
so $\ord{p}{\zeta_iq_i-1} \ge k+s.$ Then $\ord{p}{\zeta_{i+1}} \ge k+s-\ord{p}{q_i}$. Since $k>-s$, then $\ord{p}{q_i} \ge -s$. Hence, $\ord{p}{\zeta_{i+1}} \ge k+2s > \ord{p}{\zeta_i}$. Since the order of the $\zeta_i$ is strictly increasing, $|\zeta_i|_p \to 0$.
\end{proof}
\end{thm}
\begin{ex}
Let $k=1$ and consider $\xi \in \Q_7$ with $\xi^2 = \frac{1}{11}$ and $\xi \equiv 4 \mod 7$. Then $\Q(\xi)$ embeds into $\R$ by either $\psi(\xi) = \frac{1}{\sqrt{11}}$ or $\psi(\xi) = \frac{-1}{\sqrt{11}}$. For the first choice, Algorithm \ref{padicSyl} gives
$$\xi = \frac{1}{9}+\frac{7}{66}+\frac{7^3}{4709}+\frac{7^7}{72282453}+\cdots .$$
The second embedding yields
$$\xi =\frac{1}{2}+\frac{7}{12}+\frac{7^3}{617}+\frac{7^7}{1045103}+\cdots.$$
\end{ex}
Algorithm \ref{padicSyl} certainly isn't as elegant looking as Knopfmachers' Algorithm and has limitations for which $p$-adics it can be used on. However the importance of Algorithm \ref{padicSyl} is in the following theorem.
\begin{thm}\label{MainThm}
Algorithm \ref{padicSyl} terminates if and only if $\zeta \in \Q$.
\end{thm}
Instead of proving this theorem directly, in the following sections we will re-frame the approach to finding $p$-adic Sylvester expansions for rationals to mimic a more classical technique. In Section \ref{proof} we will show the link between the two and the proof of Theorem \ref{MainThm} will follow easily.

\end{section}

\begin{section}{The Fibonacci-Sylvester Greedy Algorithm}

The algorithm given by Fibonacci and Sylvester for Egyptian fractions of rationals can be interpretted 
as iterating a modified version of the classical division algorithm 
(c.f. \cite{Mays}). We review these classical methods now to  provide reference and motivation for the techniques used later.

\begin{thm}[Modifed Classical Division Algorithm]\label{DivAlg}
For all $a,b\in\mathbb{Z}$, $a>0$, there exist unique $q$, $r \in \Z$ such that 
\begin{eqnarray}
b=aq-r \label{ClassicMain}
\end{eqnarray} with  
\begin{eqnarray}
0 \le r<a. \label{ClassicBounds}
\end{eqnarray}\end{thm}

Note that Theorem \ref{DivAlg} is greedy in the sense that it finds the smallest $q$ such that $aq > b$, i.e. so that $\frac{a}{b}>\frac{1}{q}$.

\begin{alg}[F-S Greedy Algorithm] \label{ClassicalGreedy} Let $-1<\frac{a}{b} \in \Q$, with $a>0$ and $\gcd(a,b) = 1$. Iterate Theorem \ref{DivAlg} in the following way:
\begin{eqnarray}\label{iter}
\begin{aligned}
b &=& aq_0-r_0\\
bq_0 &=& r_0q_1 - r_1 \\
&\ \ \vdots&\\
bq_0q_1\cdots q_{i-1} &=& r_{i-1}q_i-r_i\\
&\ \ \vdots&
\end{aligned} \label{Algor}
\end{eqnarray}
The process terminates if any $r_N = 0$. 
\end{alg}
A straightforward computation (c.f. \cite{Mays}) then gives the following:
\begin{thm} Algorithm \ref{ClassicalGreedy} terminates in a finite number of steps and 
\begin{eqnarray}\frac{a}{b} = \sum_{i=0}^N \frac{1}{q_i}. \label{Egypt} \end{eqnarray}
\end{thm}

\end{section}
%
%
\begin{section}{A $p$-Adic Division Algorithm Approach}\label{pEgyptSec}
\subsection{The $p^k$-Division Algorithm}\label{Tony1}

We begin the process of retracing the classical approach by generalizing Theorem \ref{DivAlg}. The $p$-adic division algorithm defined here is similar to the one given in \cite{Lag} but differs considerably in the restrictions on the quotient and remainder.
\begin{thm}[$p^k$-Division Algorithm]
\label{pkDivAlg}
Let $p$ be prime and $k\in\Z$. For all $a,b \in \Z[\frac{1}{p}]$, $a>0$, there exist unique $q,r \in\mathbb{Z}[\frac{1}{p}]$ such that \begin{eqnarray}b=aq-r \label{pkmain} \end{eqnarray} with \begin{eqnarray}0\leq r< ap^k \label{pkReal}\end{eqnarray} and  \begin{eqnarray} |r|_p \le |ap^k|_p.\label{pkpAdic} \end{eqnarray}
\begin{proof}
We first prove existence. Since $\widehat{a}$ and $p$ are relatively prime, positive and negative powers of $p$ are defined mod $\widehat{a}$. Let $\alpha = \ord{p}{a}$ and $\beta = \ord{p}{b}$ and take $0\leq \overline{r}<\widehat{a}$ such that  $\overline{r}\equiv -\widehat{b}p^{\beta-\alpha-k}\mod{\widehat{a}}$.\\
(Case 1: $k>\beta-\alpha$.) For some $m \in \Z$ we have
$\overline{r}p^{-\beta+\alpha+k}+\widehat{b}=\widehat{a}m$ which gives $$\widehat{b}p^\beta=\widehat{a}p^\alpha mp^{\beta-\alpha}-\overline{r}p^{\alpha+k}.$$
Thus we can take $q=mp^{\beta-\alpha}$ and $r=\overline{r}p^{\alpha+k}$.\\
(Case 2: $k \le \beta-\alpha$.) For some $m \in \Z$ we have $\overline{r}+\widehat{b}p^{\beta-\alpha-k}=\widehat{a}m$ which gives 
\begin{eqnarray}
\widehat{b}p^\beta=\widehat{a} mp^{\alpha+k}-\overline{r}p^{\alpha+k}. \label{Case2}
\end{eqnarray}
Thus we can take $q=mp^k$ and $r=\overline{r}p^{\alpha+k}$.

In both cases, since $\overline{r} < \widehat{a}$, both (\ref{pkReal}) and (\ref{pkpAdic}) are satisfied.

To show uniqueness, suppose that there exist $q_1,r_1,q_2,r_2$ satisfying (\ref{pkmain}), (\ref{pkReal}), and (\ref{pkpAdic}). Then
\begin{equation*}
a(q_1-q_2)=r_1-r_2, \label{uni}
\end{equation*}
thus $r_1\equiv r_2\mod{\widehat{a}}$. Also $r_1\equiv r_2\mod{p^{\alpha+k}}$ since by assumption $\ord{p}{r_i}\ge \ord{p}{ap^k}$.  Therefore, $r_1\equiv r_2\mod{ap^k}$.  Since both are between $0$ and $ap^k$, $r_1 = r_2$. Thus $q_1=q_2$ and we have uniqueness as desired.
\end{proof}
\end{thm}

Note that the value $\overline{r}$ may have nontrivial order. When $\ord{p}{\overline{r}}\geq 1$ we say a jump occurs. Of course, if $p>\widehat{a}$ then $\overline{r}=\widehat{r}$ and there is no jump.


%

Also notice that one recovers Theorem \ref{DivAlg} from Theorem \ref{pkDivAlg} by setting $k=0$ (or $p=1$) and (redundantly) using the standard absolute value in (\ref{pkpAdic}). Additionally Theorem \ref{pkDivAlg} generalizes the classical algorithm in the form of the following corollary.

\begin{cor} \label{pAdicIsClassicalDivideByp}
Suppose $a,b,p, k\in\mathbb{Z}$ with $p$ prime and $k \le \ord{p}{b}-\ord{p}{a}$. Let $q_p$ denote the quotient from the $p^k$-division algorithm on $a$ and $b$ and let $q_\infty$ denote the quotient from the modified classical division algorithm on $ap^k$ and $b$. Then $q_p=q_\infty p^k$.
\end{cor}
\begin{proof}
Since $k \le \ord{p}{b}-\ord{p}{a}$,  Case 2 in the proof of Theorem \ref{pkDivAlg} applies and $q_p = mp^k$. So it suffices to show that $m = q_\infty$. By (\ref{Case2}), $$0 \le ap^km-b < ap^k.$$ Thus (\ref{ClassicMain}) and (\ref{ClassicBounds}) are satisfied.
\end{proof}

In addition, note that the $p^k$-Division Algorithm 
can be extended to rational $a$ and $b$ in the following way. Suppose $b=\frac{s}{t}$ and $a=\frac{u}{v}$ for $r,s,t,u \in \Z$. To find the quotient and remainder,  clear denominators and compute the desired division algoritm on $b' = sv$ and $a'=ut$ to find $q'$ and $r'$. Then $q = q'$ and $r = \frac{r'}{vt}$ satisfy (\ref{pkmain}), (\ref{pkReal}) and (\ref{pkpAdic}). However, since we are mostly interested in the quotient $\frac{a}{b}$, we will assume $a,b \in \mathbb{Z}[\frac{1}{p}]$.


\subsection{The $p^k$-Greedy Algorithm}\label{Tony2}

Since we now have a generalization of Theorem \ref{DivAlg}, we can substitute it into the iterative process of Algorithm \ref{ClassicalGreedy}.

\begin{alg}[$p^k$-Greedy Algorithm] \label{pkGreedy} Let $\frac{a}{b} \in \Q$, with $a>0$, $\gcd(a,b) = 1$ and $k>-\ord{p}{\frac{a}{b}}$. Iterate Theorem \ref{pkDivAlg} as in (\ref{iter}). 
The process terminates if any $r_N = 0$. 
\end{alg}

The condition $k>-\ord{p}{\frac{a}{b}}$ plays the synonymous role to $-1<\frac{a}{b}$ in the F-S Greedy Algorithm: it prevents the division algorithm from returning a quotient equal to 0. This restriction on $k$ is actually stronger than it needs to be for this alone. However, for reasons related to Corollary \ref{pAdicIsClassicalDivideByp} and explained further below, $k \le -\ord{p}{\frac{a}{b}}$ is generally undesirable. The Knopfmachers avoided this obstruction by defining their initial $a_0$ outside of the inductive pattern so that $\ord{p}{\zeta_1} \ge 1$. A similar strategy could be used here as well, though we find it more desirable to instead simply choose a different $k$ value.

\begin{thm} \label{terminates} Algorithm \ref{pkGreedy} terminates after a finite number of steps.
\end{thm}
\begin{proof} As opposed to Algorithm \ref{ClassicalGreedy}, now we have $q_i, r_i \in \Z[\frac{1}{p}]$ instead of $\Z$. However $$\widehat{r}_i \le \overline{r}_i < \widehat{r}_{i-1} \in \Z,$$ so the sequence  of $\widehat{r_i}$'s is a decreasing sequence of positive integers much like the classical remainders.
\end{proof}

Since the algorithm terminates, again we get that (\ref{Egypt}) holds.


%
\begin{ex}
Consider $\frac{a}{b}=\frac{473}{25}$. Performing the $3$-Greedy Algorithm yields,
\begin{eqnarray*}
25&=&473\cdot 2 - 921\\
25 \cdot 2 &=&921 \cdot \frac{5}{3} - 1485\\ 
25 \cdot 2 \cdot \frac{5}{3} &=& 1485\cdot \frac{115}{81} - 2025\\
25 \cdot 2 \cdot \frac{5}{3} \cdot \frac{115}{81} &=& 2025 \cdot\frac{1150}{19683} - 0
\end{eqnarray*}
and thus $$\frac{473}{25}=\frac{1}{2}+\frac{3}{5}+\frac{3^4}{115}+\frac{3^9}{1150}.$$
However, performing the $3^4$-Greedy Algorithm results in the sum
$$\frac{473}{25} = \frac{1}{23}+\frac{3^4}{5635}+\frac{3^{12}}{28175}.$$
\end{ex}

Applying Corollary \ref{pAdicIsClassicalDivideByp} to Algorithm \ref{pkGreedy} gives another relationship between the classical and $p$-adic algorithms.
\begin{cor} \label{nojumps} Suppose $\frac{a}{b}>0$ and the situation of $k \le -\ord{p}{\frac{a}{b}}$ holds. If no jumps occur than each term of the $p^k$-Greedy Algorithm for $\frac{a}{b}$ is equal to the corresponding term in the F-S Greedy Algorithm on $\frac{ap^k}{b}$ divided by $p^k$.
\end{cor}
\begin{ex}
Consider the fraction $\frac{a}{b}=\frac{5}{121}=\frac{5}{11^2}$.  If we take $p=11$ and $k=1$, then $k \le -\ord{11}{\frac{a}{b}}=2$.  Applying the $11$-Greedy Algorithm to $\frac{5}{121}$ encounters no jumps and gives $$\frac{5}{121}=\frac{1}{33}+\frac{1}{99}+\frac{1}{1089}.$$ On $\frac{5\cdot 11}{121} = \frac{5}{11}$, the F-S Greedy Algorithm yields $$\frac{5}{11}=\frac{1}{3}+\frac{1}{9}+\frac{1}{99}.$$
\end{ex}

The restriction on jumps is sufficient but not necessary. For example, the $3$-Greedy Algorithm on $\frac{22}{45}$ encounters a jump, yet the relation above to the classical algorithm on $\frac{22}{15}$ still holds.

Also notice that, again with $k \le -\ord{p}{\frac{a}{b}}$, if $\frac{ap^k}{b} >1$ then the F-S Greedy Algorithm returns quotients $q_i=1$ until the remaining value to be decomposed is less than 1. Hence, by Corollary \ref{nojumps} the $p^k$-Greedy Algorithm on $\frac{a}{b}$ produces a finite string of terms equal to $\frac{1}{p^k}$. For these reasons we only consider the cases where $k>-\ord{p}{\frac{a}{b}}$. 

This condition is minor, though, in the grand scheme. Each inductive step of Algorithm \ref{pkGreedy} is independent of each other with respect to the value of $k$. So if the desired $k$ fails the order criteria, it is possible to temporarily use a sufficiently large $k'$ in the initial step(s) and then switch back to the desired $k$ value once the corresponding orders become large enough. In this way finite $p$-adic Sylvester expansions can be found for all rationals.

\subsection{Connection Between Approaches}\label{proof}

We are now in the position to return to the modification of the Knopfmachers' algorithm given in Section \ref{padic}.

\begin{thm}\label{same}
For $\zeta \in \Q$, Algorithm \ref{padicSyl} and Algorithm \ref{pkGreedy} produce the same output.
\begin{proof}
Suppose $\zeta_i \in \Q$ for some $i \ge 0$. Then $\zeta_i = \frac{a}{b}$ for $a$, $b\in\Z$. Algorithm \ref{padicSyl} produces $$q_i= \left\langle \frac{b}{a} \right\rangle_k+\left\lceil\frac{b/a-\langle b/a\rangle_k}{p^k}\right\rceil p^k.$$ Let $r=aq_i-b$ and $\sigma_k=\left\langle \frac{b}{a} \right\rangle_k-\frac{b}{a}$. Then 
$$0 \le \frac{r}{a} = \sigma_k + \left\lceil\frac{-\sigma_k}{p^k}\right\rceil p^k<p^k.$$
Since $\ord{p}{\sigma_k} \ge k$, then $|r|_p \le |ap^k|_p$. Hence both (\ref{pkReal}) and (\ref{pkpAdic}) are satisfied.
\end{proof}
\end{thm}
The proof of Theorem \ref{MainThm} is simply a combination of Theorem \ref{terminates} and Theorem \ref{same}. 
Further the condition of $k >-\ord{p}{\zeta}$ in Algorithm \ref{padicSyl} is explained in light of Corollary \ref{nojumps}. Indeed, for negative or irrational $\zeta$ with $k \le -\ord{p}{\zeta}$, a statement analogous to Corollary \ref{nojumps} holds and thus Algorithm \ref{padicSyl} produces a $p$-adically divergent result. Though, as mentioned above, this can be worked around by temporarily using alternate $k$-values.

\end{section}

\subsection*{Acknowledgements}
Thank you to Anthony Martino for being the first to convince me that Egyptian fractions are interesting and for helping me work out parts of Sections \ref{Tony1} and \ref{Tony2} in the $k=1$ case. 

This work is dedicated to my children, Gedion and Eomji.

%
\end{document}